\documentclass[11pt]{amsart} %

\usepackage{float}

\usepackage{epsfig}

\usepackage{array,amsmath, enumerate,  url, psfrag}
\usepackage{amssymb, amsaddr, fullpage}
\usepackage{graphicx,subfigure}
 \usepackage[usenames,dvipsnames]{pstricks}
\usepackage{color}

\restylefloat{figure}
 \makeatother

\newtheorem{theorem}{Theorem}[section]
\newtheorem{lemma}[theorem]{Lemma}
\newtheorem{conjecture}[theorem]{Conjecture}

\newtheorem{observation}[theorem]{Observation}

\theoremstyle{definition}

\title{Partitioning planar graphs without 4-cycles and 6-cycles 
into a forest and a disjoint union of paths}
\author{Pongpat Sittitrai$^{1}$ \hskip 0.2in Kittikorn Nakprasit$^{2}$}

\address{
	$^{1}$\small Department of Mathematics, Faculty of Science, Khon Kaen University, Khon Kaen, 40002, Thailand.\newline
	Email : pongpat.sittitrai@gmail.com\newline
	$^{2}$\small Department of Mathematics, Faculty of Science, Khon Kaen University, Khon Kaen, 40002, Thailand.\newline
	Email : kitnak@hotmail.com}
\begin{document}

\maketitle

\begin{center}{\bf Abstract}\end{center}
\indent\indent

In this paper, we show that every planar graph without $4$-cycles 
and $6$-cycles has a partition of its vertex set into two sets, 
where one set induces a forest, and the other induces a forest 
with maximum degree at most $2$ 
(equivalently, a disjoint union of paths). 

Note that we can partition the vertex set of a forest into two independent sets. 
However a pair of independent sets combined may not induce a forest. 
Thus our result extends the result of Wang and Xu (2013) 
stating that the vertex set of every planar graph without $4$-cycles 
and $6$-cycles can be partitioned into three sets, 
where one induces a graph with maximum degree two, 
and the remaining two are independent sets.

\section{Introduction} 
In this paper, we consider only undirected simple graphs. 
Let $\mathcal{G}_i$ be a family of graphs. 
A graph $G$ with the vertex set $V(G)$ has a 
\emph{$(\mathcal{G}_1, \dots,  \mathcal{G}_k)$-partition} 
$(V_1, \dots, V_k)$ (a vertex partition) 
if $V(G)$ can be partitioned into 
$k$ sets $V_1$, $V_2,\dots, V_k$ where $V_i$ is an empty set or 
the induced subgraph $G[V_i]$ is in $\mathcal{G}_i$ for each $i\in\{1,\dots,k\}$. 


Certain classes of graphs are of interest. 
Let  $\mathcal{F}_d$ be a family of forests with maximum degree $d$, 
and let  $\Delta_d$ be a family of graphs with maximum degree $d$. 
We use $\mathcal{I}$ for $\mathcal{F}_0$ and $\Delta_0$, 
and we use $\mathcal{F}$ for $\mathcal{F}_\infty$ 
(a family of forests with unbounded degree). 

Note that a $(\Delta_{d_1},\ldots,\Delta_{d_k})$-partition is 
equivalent to a $(d_1,d_2,\dots, d_n)$-coloring. 
Accordingly, an improper vertex coloring 
(a generalization of a proper coloring) 
can be regarded as a kind of vertex partition 


 
The following table shows some known results about the existence of 
particular vertex partitions for some classes of planar graphs. 
Some results may be redundant since a forest  
can be partitioned into one or two independent sets 
and $\mathcal{F}_d$ is a subclass of $\Delta_d.$ 
Nonetheless, we still put original results about 
$(\Delta_{d_1}, \Delta_{d_2},\dots,  \Delta_{d_n})$-partition 
in the table for a chronological reason.

\begin{center}
	\begin{tabular}{ |c|c|c| } 
		\hline
		Classes of Planar graphs  & $(\Delta_{d_1}, \dots,  \Delta_{d_n})$-partition & $(\mathcal{F}_{d_1}, \dots,  \mathcal{F}_{d_n})$-partition \\ 
	\hline
	Planar graphs& $(\mathcal{I},\mathcal{I},\mathcal{I},\mathcal{I})$ [FCT] & $(\mathcal{I},\mathcal{F},\mathcal{F})$ \cite{IFF} \\ 
		 & $(\Delta_2,\Delta_2,\Delta_2)$ \cite{222}& $(\mathcal{F}_2,\mathcal{F}_2,\mathcal{F}_2)$ \cite{F2F2F2} \\ 
	\hline
	Planar graphs with girth 4& $(\mathcal{I},\mathcal{I},\mathcal{I})$ \cite{000C3} & $(\mathcal{F}_5,\mathcal{F})$ \cite{F5FG4} \\ 
	\hline
	Planar graphs with girth 5& $(\Delta_3,\Delta_4)$\cite{34G5}  & $(\mathcal{F}_3,\mathcal{F}_3)$ (forbid adjacent $5$-cycles) \cite{F33G5} \\ 
	&  & $(\mathcal{I},\mathcal{F})$ \cite{IFG5} \\ 
	\hline
	Planar graphs with girth 6& $(\Delta_1,\Delta_4)$  \cite{02o04o14G8} & $(\mathcal{F}_1,\mathcal{F}_4)$ \cite{F1F4G6}  \\ 
	& $(\Delta_2,\Delta_2)$ \cite{22G6} & $(\mathcal{F}_2,\mathcal{F}_2)$ \cite{F2F2G6} \\ 
	\hline
	Planar graphs with girth 7& $(\mathcal{I},\Delta_4)$ \cite{02o04o14G8} & $(\mathcal{I},\mathcal{F}_5)$  \cite{IF235}\\ 
	\hline
	Planar graphs with girth 8& $(\mathcal{I},\Delta_2)$ \cite{02o04o14G8} & $(\mathcal{I},\mathcal{F}_3)$ \cite{IF235} \\ 
	\hline
	Planar graphs with girth 10& $(\mathcal{I},\mathcal{I},\mathcal{I})$ \cite{IF235} & $(\mathcal{I},\mathcal{F}_2)$ \cite{IF235}  \\ 
	\hline
	\end{tabular}
\end{center}

In \cite{giveva3}, Chartrand and Kronk gave an example of 
a planar graph without an $(\mathcal{F},\mathcal{F})$-partition. 
In view of this, finding the sufficient conditions for planar graphs 
to have an $(\mathcal{F},\mathcal{F})$-partition has become 
an interesting topic ever since. 

Every planar graphs with girth $4$ has an $(\mathcal{F},\mathcal{F})$-partition 
by being  $2$-degenerate. 
On the other hand, Montassier and  Ochem \cite{NomnC3} 
showed that for each $d_1$ and $d_2$, there exists 
a planar graph with girth $4$ having no 
$(\Delta_{d_1},\Delta_{d_2})$-partitions. 
Thus the result of  partitioning the vertex set of  
a planar graph with girth $4$ into two forests cannot be improved 
in terms of the maximum degrees of both forests. 
At best, one may find $d_1$ such that each planar graph 
with girth $4$ has an $(\mathcal{F}_{d_1},\mathcal{F})$-partition. 
Dross et al. \cite{F5FG4} verified this holds for $d_1=5$. 
However the case for $d_1 \leq 4$ is still open. 
In particular, the case $d_1=0$ if true (an $(\mathcal{I},\mathcal{F})$-partition) 
would imply the result by Gr\"otzsch \cite{000C3} 
(a $(\mathcal{I},\mathcal{I},\mathcal{I})$-partition). 

A planar graph without $4$- and $6$-cycles is shown to have 
an $(\mathcal{I},\mathcal{I},\mathcal{I})$-partition 
if it has no $8$-cycles by Wang and Chen \cite{000C4C6C8}  
or it has no $9$-cycles by Kang et al. \cite{000C4C6C9}. 
Liu and Yu \cite{1FC4C6C8C9} improved both results by showing 
that such graphs have an  $(\mathcal{I},\mathcal{F})$-partition. 

Sittitrai and Nakprasit \cite{35o44C4C5} proved that 
every planar graph without $4$-cycles and $5$-cycle has 
a $(\Delta_4,\Delta_4)$-partition and a $(\Delta_3,\Delta_5)$-partition. 
Later Cho et al. \cite{F34C45}  extended the result by proving that 
each such graph has an $(\mathcal{F}_3,\mathcal{F}_4)$-partition.
 
Wang and Xu \cite{200C4C6} showed that every planar graphs without 
$4$-cycles and $6$-cycles has a $(\Delta_2,\mathcal{I},\mathcal{I})$-partition. 
In this work, we improve their result in the following theorem.  
 
 
 \begin{theorem}\label{main}
 	Every planar graph without $4$-cycles and $6$-cycles has 
 	an ($\mathcal{F}_2,\mathcal{F}$)-partition.
 \end{theorem} 
 
On the other hand, in 2022 Kang et al. \cite{100C4C6} showed that 
every planar graph without $4$-cycles and $6$-cycles has 
an $(\mathcal{F}_1,\mathcal{I},\mathcal{I})$-partition 
(equivalently, a $(1,0,0)$-coloring). 
Inspired by two above results, we put forth the following conjecture.    
\begin{conjecture} 	Every planar graph without $4$-cycles and $6$-cycles 
has an $(\mathcal{F}_1,\mathcal{F})$-partition.
\end{conjecture}

\section{Structures of a minimal counterexample}

Before we proceed to accrue results, 
some notation is required as follows. 
A \emph{$k$-vertex} (respectively, \emph{$k^+$-vertex} 
and  \emph{$k^-$-vertex})  is a vertex of degree $k$ 
(respectively, at least $k$ and at most $k.$) 
The same notation is applied for faces. 
A $k$-vertex $u$ is a \emph{$k$-neighbor} of $v$ if 
$u$ is adjacent to $v.$ 
The boundary walk of a face $f$ is denoted by $b(f).$ 
A vertex $v$ and a face $f$ are incident if $v$ is on $b(f).$ 
If a vertex $v$ not on $b(f)$ but is adjacent to a $3$-vertex 
$u$ on $b(f),$  then we call $f$ a \emph{pendent face} 
of a vertex $v$ and  $v$ is a \emph{pendent neighbor} 
of $u$ (with respect to $f$).   
We use $n_i(v)$ to denote the number of incident $i$-faces 
of a vertex $v$, 
and use $m_i(v)$ to denote the number 
of pendent  $i$-faces of a vertex $v$.

Given a $3$-vertex $u$ incident to a $3$-face or a $5$-face $f$, 
we call $u$ a \emph{terrible $3$-vertex} of $f$ if it has a 
pendent $4^-$-neighbor, otherwise we call $u$ 
a \emph{non-terrible $3$-vertex}. 
A $3$-face $f$ is a \emph{poor $3$-face} 
if $f$ is incident to two terrible $3$-vertices. 

Let $G$ be a minimal counterexample of Theorem~\ref{main}. 
That is $G$ does not has a ($\mathcal{F}_2,\mathcal{F}$)-partition, 
but each proper subgraph $G'$ of $G$ has 
an ($\mathcal{F}_2,\mathcal{F}$)-partition $(V_1,V_2).$  
For a vertex $v\in V_1,$  
a neighbor of $v$ in $V_1$ is a \emph{$V_1$-neighbor}, 
and we say $v$ is \emph{$V_1$-saturated} 
if $v$ has two $V_1$-neighbors. 
Some properties of $G$ are obtained as follows.

\begin{lemma}\label{atleast3} Each vertex in $G$ is a $3^+$-vertex.
\end{lemma}
\begin{proof} Let $u$ be a vertex in $G$. 
Suppose to the contrary that $u$ is a $2$-vertex. 
(The case that $u$ is a $1^-$-vertex is similar). 
Then $G-v$ has a ($\mathcal{F}_2,\mathcal{F}$)-partition $(V_1,V_2)$.  
Let $v$ and $w$ are two adjacent vertices of $u.$ 
	
If $v\in V_1$ or $w\in V_1$, 
then $G$ has a ($\mathcal{F}_2,\mathcal{F}$)-partition 
$(V_1,V_2\cup \{u\}),$ a contradiction. 
	
If $v, w\in V_2$, then $G$ has a ($\mathcal{F}_2,\mathcal{F}$)-partition $(V_1\cup \{u\},V_2),$ a contradiction.
\end{proof}

Since $G$ contains neither $4$- nor $6$-cycles, 
we have the following observation. 

\begin{observation}\label{face} Let $f$ be a face in $G$. 
	\begin{enumerate}
		\item [\rm (i)] A face $f$ is not a $4$-face.  
		\item [\rm (ii)] If $f$ is a $3$-face, 
		then $f$ is not adjacent to a $6^-$-face.
	\end{enumerate}
\end{observation}

\begin{lemma}\label{forbidvertex} 
    Let $v$ be a vertex of $G.$ 
	\begin{enumerate}
		\item [\rm (i)] If $v$ is a $3$-vertex, 
		        then $v$ is adjacent to a $5^+$-vertex. 
		\item [\rm (ii)] If $v$ is a $5$-vertex, 
		        then $v$ is not incident to a poor $3$-face. 
		\item [\rm (iii)] If $v$ is a $6$-vertex, 
		        then $v$ is incident to at most one poor $3$-face. 
	\end{enumerate}	
\end{lemma}

\begin{proof} 
	
	\noindent (i)  Suppose to the contrary that 
	each neighbor of $v$ is a $4^-$-vertex. 
	Consider ($\mathcal{F}_2,\mathcal{F}$)-partition $(V_1,V_2)$ 
	of $G-v.$ 
	Since each neighbor of $v$ is a $3^-$-vertex in $G-v$, 
	we may assume that each of them is in $V_1$ 
	but is not $V_1$-saturated or is in $V_2.$    
	
	If two or three neighbors of $v$ are in $V_1,$ 
	then $G$ has a ($\mathcal{F}_2,\mathcal{F}$)-partition 
	$(V_1,V_2\cup \{v\}),$ 
	otherwise  $G$ has an ($\mathcal{F}_2,\mathcal{F}$)-partition $(V_1\cup \{v\},V_2)$ since the only neighbor of $v$ in $V_1$ 
	(if exists) is not $V_1$-saturated. 
	We obtain a contradiction for both cases.
	
    (ii)  Suppose to the contrary that $v$ is incident 
    to a poor $3$-face $f$ with $b(f)= vv_1v_2.$ 
    From the definition, $v_1$ and  $v_2$ are terrible 
    $3$-vertices with pendent $4^-$-neighbors, 
    say $v'_1$ and  $v'_2$, respectively. 
	
	Consider $G-\{v_1,v_2\}$ with an  ($\mathcal{F}_2,\mathcal{F}$)-partition $(V_1,V_2).$  
	
	Since $v$, $v'_1$, and $v'_2$ are $3^-$-vertices in $G-\{v_1,v_2\}$, 
	we may assume that each of them is in $V_1$ 
	but is not $V_1$-saturated or is in $V_2.$  
	
	-	Let $v\in V_1$. 
	
	If $v'_1$ or $v'_2$ is in $V_1$, 
	then $G$ has a ($\mathcal{F}_2,\mathcal{F}$)-partition
	$(V_1,V_2\cup\{v_1,v_2\}),$  a contradiction.
	
	If $v'_1$ and $v'_2$ are in $V_2$, 
	then $G$ has a ($\mathcal{F}_2,\mathcal{F}$)-partition 
	$(V_1\cup\{v_1\},V_2\cup\{v_2\})$ 
	since $v$ is not $V_1$-saturated, a contradiction.

	-	Let $v\in V_2$. 
	
	If $v'_1$ and $v'_2$ are in $V_2$, then $G$ has a ($\mathcal{F}_2,\mathcal{F}$)-partition $(V_1\cup\{v_1,v_2\},V_2),$  a contradiction.
	
	If $v'_1$ is in $V_1$ and $v'_2$ is in $V_2$, 
	then $G$ has a ($\mathcal{F}_2,\mathcal{F}$)-partition 
	$(V_1\cup\{v_2\},V_2\cup\{v_1\}),$ a contradiction.
	
	If $v'_1$ and $v'_2$ are in $V_1$, 
	then $G$ has a ($\mathcal{F}_2,\mathcal{F}$)-partition 
	$(V_1\cup\{v_1\},V_2\cup\{v_2\})$ 
	since $v'_1$ is not $V_1$-saturated, a contradiction. 
	
	(iii)  Suppose to the contrary that 
	$v$ is incident to two poor $3$-faces
	with boundary walks $vv_1v_2$ and $vv_3v_4.$  
	From the definition,  $v_i$ is a terrible $3$-vertices 
	with pendent $4^-$-neighbors, say $v'_i$ 
	where $i\in \{1,2,3,4\}$.   
	
	Consider $G'=G-\{v_1,v_2,v_3,v_4\}$ with 
	a ($\mathcal{F}_2,\mathcal{F}$)-partition $(V_1,V_2).$  
	Since each $v'_i$ is a $3^-$-vertex	in $G'$, 
	we may assume that $v'_i\in V_1$ but 
	$v'_i$ is not $V_1$-saturated or $v'_i\in V_2$  
	where $i\in \{1,2,3,4\}.$ 
	Moreover, since $v$ is a $2$-vertex in $G',$ 
	we may assume that $v\in V_1$ 
	but has no $V_1$-neighbors or $v\in V_2.$ 
	The table below shows that $G$ has 
	an ($\mathcal{F}_2,\mathcal{F}$)-partition for all cases, 
	a contradiction. 

\begin{center}
	\begin{tabular}{ |c|c|c|c| } 
		\hline
		$v$ & $v_1$ and $v_2$ &  $v_3$ and $v_4$  
		& $(\mathcal{F}_{2},  \mathcal{F})$-partition \\ 
	\hline
	$v\in V_1$ & $v'_1 \in V_1$ or $v'_2 \in V_1$ 
	&$v'_3 \in V_1$ or $v'_4 \in V_1$
	&$(V_1,V_2\cup\{v_1,v_2,v_3,v_4\})$\\ 
	\hline
	$v\in V_1$ & $v'_1 \in V_1$ or $v'_2 \in V_1$ 
	&$v'_3 \in V_2$ and $v'_4 \in V_2$
	&$(V_1\cup\{v_4\},V_2\cup\{v_1,v_2,v_3\})$\\ 
	\hline
	$v\in V_1$ & $v'_1 \in V_2$ and $v'_2 \in V_2$ 
	&$v'_3 \in V_2$ and $v'_4 \in V_2$
	&$(V_1\cup\{v_2,v_4\},V_2\cup\{v_1,v_3\})$\\
	\hline

	$v\in V_2$ & $v'_1 \in V_1$ or $v'_2 \in V_1$ 
	&$v'_3 \in V_1$ or $v'_4 \in V_1$
	&$(V_1\cup\{v_1,v_2,v_3,v_4\},V_2)$\\ 
	\hline
	$v\in V_2$ & $v'_1 \in V_1$ or $v'_2 \in V_1$ 
	&$v'_3 \in V_2$ and $v'_4 \in V_2$
	&$(V_1\cup\{v_1,v_2,v_3\},V_2\cup\{v_4\})$\\ 
	\hline
	$v\in V_2$ & $v'_1 \in V_2$ and $v'_2 \in V_2$ 
	&$v'_3 \in V_2$ and $v'_4 \in V_2$
	&$(V_1\cup\{v_1,v_2,v_3\},V_2\cup\{v_2,v_4\})$\\ 
	\hline
	\end{tabular}
\end{center}

\end{proof}

\begin{lemma}\label{number}
Each $k$-vertex $v$ has upper bounds on $m_i(v)$ and $n_i(v)$ 
as follows. 

	\begin{enumerate}
		\item [\rm (i)] $n_3(v)\leq\lfloor\frac{k}{2}\rfloor$. 
		
		\item [\rm (ii)]  
		$
		n_5(v) \leq
		\begin{cases}
		0,        & \text{if } m_3(v)+2n_3(v)= k,\\
		k,     & \text{if } m_3(v)+2n_3(v)=0,\\
		k-m_3(v)-2n_3(v)-1        & \text{,otherwise}.\\
		\end{cases}
		$

		\item [\rm (iii)]  $m_3(v)\leq k-2 n_3(v)$.
		
		\item [\rm (iv)] $m_5(v)\leq k-m_3(v)-2m_3(v)$.
	\end{enumerate}	
\end{lemma}	
\begin{proof} 
	Let $v_1, \ldots, v_k$ be neighbors of a $k$-vertex $v.$
	
	Let $A$ be the set of  $v_i$  where 
	$v_i$ is incident to an incident $3$-face of $v$. 
	
	Let $B$ be the set of  $v_i$ where 
	$v_i$ is pendent $3$-neighbor of $v$. 
	
	Let $C$ be the set of  $v_i$  where 
	$v_i$ is incident to an incident $5$-face of $v$. 
	
	Let $D$ be the set of  $v_i$  where 
	$v_i$ is a pendent $5$-neighbor of $v$.
	
	Recall Observation~\ref{face}(ii) that 
	a $3$-face is not adjacent to a $6^-$-face, 
	we have the following properties.
	
	(a) $A\cap B=\emptyset$ and  
	$(A\cup B)\cap (C\cup D)=\emptyset.$ 
	Consequently each $|C|$ and $|D|$  
	is not more than $k-|A|-|B|$. 
	
	(b) If $v_i\in A$, then $v_i$ is incident to exactly one incident $3$-face of $v$. 
	Consequently $n_3(v) \leq \frac{|A|}{2}.$ 
	
	(c) If $v_i\in C$, then  $v_{i-1}$ or $v_{i+1}$ is in $C$. 
	Consequently $n_5(v)= |C|-1$ for $1\leq |C| \leq k-1.$
	
	The Lemma \ref{number} follows (a), (b), and (c).
\end{proof}

\section{Proof of Theorem \ref{main}}

\indent Suppose $G$ is a minimal counterexample to the theorem. 
The discharging process is as follows. 
Let the initial charge of a vertex $v$ in $G$ be $\mu(v)=2d(v)-6,$ 
and let the initial charge of a face $f$ in $G$ be $\mu(f)=d(f)-6$. 
Using Euler's formula $|V(G)|-|E(G)|+|F(G)|=2$ 
and  the Handshaking lemma, we have
$$\displaystyle\sum_{v\in V(G)}\mu(v)
+\displaystyle\sum_{f\in F(G)}\mu(f)=-12.$$

\indent Now, we establish a new charge $\mu^*(x)$ 
for all $x\in V(G)\cup F(G)$ by transferring charge 
from one element to another and 
the summation of new charge $\mu^*(x)$ remains $-12.$  
If the final charge  $\mu^*(x)\geq 0$ for all $x\in V(G)\cup F(G)$, 
then we get a contradiction and the proof is completed.\\
\indent The discharging rules are as follows.\\
\noindent (R1) Let $v$ be a $3$-vertex. \\
- $v$ gives charge $\frac{1}{3}$ to each incident $3$-face. \\
\noindent (R2) Let $v$ be a $4$-vertex. \\
- $v$ gives charge $1$ to each incident $3$-face. \\
- $v$ gives charge $\frac{1}{2}$ to each incident $5$-face.\\
\noindent (R3) Let $v$ be a $5^+$-vertex. \\
- $v$ gives charge $\frac{5}{3}$ to each incident non poor $3$-face or $\frac{7}{3}$ to  each incident poor $3$-face. \\
- $v$ gives charge $\frac{2}{3}$ to each pendent $3$-face.\\
- $v$ gives charge $\frac{1}{4}$ to each pendent $5$-face.\\
- $v$ gives charge $\frac{1}{2}$ to each incident $5$-face.\\
\noindent (R4) Let $f$ be a $7^+$-face. \\
- $f$ gives charge $\frac{1}{6}$ to 
each incident $3$-vertex incident to a $3$-face.

\indent It remains to show that resulting $\mu^*(x)\geq 0$ 
for all $x\in V(G)\cup F(G)$. 
It is clear that $\mu^*(x)\geq 0$ 
when $x$ is a $3$-vertex not incident to any $3$-faces 
or $x$ is a $6$-face. 

\noindent\textit{\textbf{CASE 1:}} Let $f$ be a $3$-face.\\
\indent Note that $f$ is incident to a $5^+$-vertices 
if $f$ is incident to a terrible $3$-vertex 
by Lemma \ref{forbidvertex}(i). 
It follows that $f$ has at most two incident terrible $3$-vertices. 

Let $k$ denote the number of incident non-terrible 
$3$-vertices of $f.$ 

\indent  If $f$ has no incident terrible $3$-vertices,  
then $\mu^*(f)= \mu(f)+(3-k)\times1
+n\times\frac{1}{3}+k\times\frac{2}{3}=0$ 
by (R1), (R2), and (R3).

\indent  If $f$ has one incident terrible $3$-vertex,  
then $\mu^*(f)= \mu(f)+\frac{1}{3}+\frac{5}{3}+(1-k)\times1 
+k\times\frac{1}{3}+k\times\frac{2}{3}=0$ by (R1), (R2), and (R3). 

\indent  If $f$ has two incident terrible $3$-vertices,   
then $\mu^*(f)= \mu(f)+2\times\frac{1}{3}+\frac{7}{3}=0$ 
by (R1) and (R3).

\noindent\textit{\textbf{CASE 2:}} Let $f$ be a $5$-face.

\indent  If $f$ has no incident $4^+$-vertices,  
then $f$ has five non-terrible $3$-vertices 
by Lemma \ref{forbidvertex} (i). 
Thus $\mu^*(f)= \mu(f)+5\times\frac{1}{4}>0$ by (R3).

\indent  If $f$ has an incident $4^+$-vertex,  
then $f$ has at least  two non-terrible $3$-vertices 
by  Lemma \ref{forbidvertex} (i).
Thus $\mu^*(f)= \mu(f)+2\times\frac{1}{4}+\frac{1}{2}\geq 0$ by  (R2) and (R3).

\indent  If $f$ has at least two incident $4^+$-vertices, 
then $\mu^*(f)= \mu(f)+2\times\frac{1}{2}\geq 0$ by (R2) and (R3).

\noindent\textit{\textbf{CASE 3:}} Let $f$ be a $7^+$-face.

Note that a $7$-face is incident to at most six $3$-vertices 
incident to a $3$-face by Observation \ref{face}. 
By (R4), $\mu^*(f)= \mu(f)-6\times\frac{1}{6}\geq 0$ 
for a $7$-face $f$, 
and $\mu^*(f)= \mu(f)-d(v)\times\frac{1}{6}>0$ 
for a $8^+$-face $f.$ 

\noindent\textit{\textbf{CASE 4:}} 
Let $v$ be a $3$-vertex incident to a $3$-face. 

By Observation \ref{face},  
$v$ incident to one $3$-face and two $7^+$-faces. 
Thus $\mu^*(v)= \mu(v)-\frac{1}{3}+2\times\frac{1}{6}=0$ 
by  (R1) and (R4).

\noindent\textit{\textbf{CASE 5:}} Let $v$ be a $4$-vertex. 

By (R2),  $v$ loses charge $n_3(v)\times1+n_5(v)\times\frac{1}{2}$. 
By Lemmas \ref{number} (i) and (ii), we have the following cases. 

If $n_3(v)=0$, then $n_5(v)\leq4$. 
Thus $\mu^*(v)\geq \mu(v)-4\times\frac{1}{2}=0.$

If $n_3(v)=1$, then $n_5(v)\leq1$. 
Thus $\mu^*(v)\geq \mu(v)-1\times1-1\times\frac{1}{2}=0.$

If $n_3(v)=2$, then $n_5(v)=0$. 
Thus $\mu^*(v)= \mu(v)-2\times1=0.$ 

\noindent\textit{\textbf{CASE 6:}} Let $v$ be a $5^+$-vertex and $2\times n_3(v)+m_3(v)=0$. 

By Lemma \ref{number} (ii) and (iv), $n_5(v)\leq d(v)$ and $m_5(v)\leq d(v).$

Thus, 
$\mu^*(v) \geq \mu(v)-d(v)\times\frac{1}{2}-d(v)\times\frac{1}{4}
=2d(v)-6-d(v)\times \frac{3}{4}=\frac{5}{4}\times d(v)-6>0$ 
by (R3) and $d(v)\leq 5$.

\noindent\textit{\textbf{CASE 7:}} 
Let $v$ be a $5$-vertex with $2\times n_3(v)+m_3(v)>0$. 

By Lemma \ref{forbidvertex} (ii), 
$v$ is not incident  to a poor $3$-face. 
Then $v$ gives charge $\frac{5}{3}$ to each incident $3$-face. 
By (R3), we have\\
$$
\mu^*(v)=2d(v)-6-n_3(v)\times\frac{5}{3}+m_3(v)\times\frac{2}{3}+n_5(v)\times\frac{1}{2}+m_5(v)\times\frac{1}{4}.
$$

\textit{\textbf{CASE 7.1:}} 
Suppose $2\times n_3(v)+m_3(v)=d(v)=5$. 

By Lemma \ref{number}, we have $n_3(v)\leq2$ and  $m_3(v)= d(v)-2\times n_3(v),$ and $n_5(v)=m_5(v)=0.$
\begin{align*}
\mu^*(v)&=2d(v)-6-n_3(v)\times\frac{5}{3}-(d(v)-2\times n_3(v))\times\frac{2}{3}\\
&=\frac{4}{3}\times d(v)-6-n_3(v)\times\frac{1}{3}\\
&\geq\frac{4}{3}\times 5-6-2\times\frac{1}{3}\\
&= 0.
\end{align*}

\textit{\textbf{CASE 7.2:}} 
Suppose $0<2\times n_3(v)+m_3(v)<d(v)=5$.

By Lemma \ref{number}, we have $n_3(v)\leq2$,  
$n_5(v)\leq d(v)-2n_3(v)-m_3(v)-1$, 
and $m_5(v)\leq d(v)-2\times n_3(v)-m_3(v).$

\begin{align*}
\mu^*(v)&\geq 2d(v)-6-n_3(v)\times\frac{5}{3}-m_3(v)\times\frac{2}{3}-(d(v)-2n_3(v)-m_3(v)-1)\times\frac{1}{2}\\
& -(d(v)-2n_3(v)-m_3(v))\times\frac{1}{4}\\
&=\frac{5}{4}\times d(v)-\frac{11}{2}-n_3(v)\times\frac{1}{6}+m_3(v)
\times\frac{1}{12}\\
&\geq \frac{5}{4}\times d(v)-\frac{11}{2}-n_3(v)\times\frac{1}{6}\\
&\geq \frac{5}{4}\times 5-\frac{11}{2}-2\times\frac{1}{6}\\
&> 0.
\end{align*}

\noindent\textit{\textbf{CASE 8:}} Let $v$ be a $6$-vertex 
with $2\times n_3(v)+m_3(v)>0$. 

By Lemma \ref{forbidvertex} (iii), 
$v$ is incident to at most one poor $3$-face. 
Note that it is enough to consider only 
the case $v$ containing exactly one poor $3$-face. 

By (R3),  we have   
\begin{align*}
\mu^*(v)&=2d(v)-6-\frac{7}{3}-(n_3(v)-1)\times\frac{5}{3}
-m_3(v)\times\frac{2}{3}-n_5(v)\times\frac{1}{2}
-m_5(v)\times\frac{1}{4}\\
&=-\frac{2}{3}+(2d(v)-6-n_3(v)\times\frac{5}{3}
-m_3(v)\times\frac{2}{3}-n_5(v)\times\frac{1}{2}
-m_5(v)\times\frac{1}{4}).
\end{align*}

\textit{\textbf{CASE 8.1:}} Suppose $2\times n_3(v)+m_3(v)=d(v)=6$. 

By Lemma \ref{number} (iii) and (iv), $n_5(v)=m_5(v)=0.$ 

It follows from \textit{\textbf{CASE 7.1}} that
\begin{align*}
\mu^*(v)&=-\frac{2}{3}+(\frac{4}{3}\times d(v)-6-n_3(v)\times\frac{1}{3})\\
&=\frac{4}{3}\times d(v)-\frac{20}{3}-n_3(v)\times\frac{1}{3}\\
&\geq\frac{4}{3}\times 6-\frac{20}{3}-3\times\frac{1}{3}\\
&>0.
\end{align*}

\textit{\textbf{CASE 8.2:}} Let $0<2\times n_3(v)+m_3(v)<d(v)=6$.

It follows from \textit{\textbf{CASE 7.2}} that
\begin{align*}
\mu^*(v)&\geq-\frac{2}{3}+(\frac{5}{4}\times d(v)
-\frac{11}{2}-n_3(v)\times\frac{1}{6})\\
&=\frac{5}{4}\times d(v)-\frac{37}{6}-n_3(v)\times\frac{1}{6}\\
&\geq\frac{5}{4}\times 6-\frac{37}{6}-3\times\frac{1}{6}\\
&>0.
\end{align*}

\noindent\textit{\textbf{CASE 9:}} Let $v$ be a $7^+$-vertex 
with $2\times n_3(v)+m_3(v)>0$. 

Then $v$ gives charge at most $\frac{7}{3}$ to each incident $3$-face. 

By R(3),  we have 
\begin{align*}
\mu^*(v)&=2d(v)-6-n_3(v)\times\frac{7}{3}-m_3(v)\times\frac{2}{3}
-n_5(v)\times\frac{1}{2}-m_5(v)\times\frac{1}{4}\\
&=-n_3(v)\times\frac{2}{3}+(2d(v)-6-n_3(v)\times\frac{5}{3}
-m_3(v)\times\frac{2}{3}-n_5(v)\times\frac{1}{2}
-m_5(v)\times\frac{1}{4}).
\end{align*}

\textit{\textbf{CASE 9.1:}} Suppose $2\times n_3(v)+m_3(v)=d(v)$.

By Lemma \ref{number} (iii) and (iv), we have $n_5(v)=m_5(v)=0.$ 

It follows from \textit{\textbf{CASE 7.1}} that
\begin{align*}
\mu^*(v)&=-n_3(v)\times\frac{2}{3}+(\frac{4}{3}\times d(v)-6-n_3(v)\times\frac{1}{3})\\
&=\frac{4}{3}\times d(v)-6-n_3(v)\times1.
\end{align*}

If $v$ is a $7$-vertex, then $d(v)=7$ and $n_3(v)\leq3$. 
Thus $\mu^*(v)\geq\frac{4}{3}\times 7-6-3>0.$

If $v$ is a $8^+$-vertex, then $d(v)\geq 8$ 
and $n_3(v)\leq\lfloor\frac{d(v)}{2}\rfloor\leq \frac{d(v)}{2}$. 
Thus $\mu^*(v)\geq\frac{4}{3}\times d(v)-6-\frac{d(v)}{2}>0.$

\textit{\textbf{CASE 9.2:}} Suppose $0<2\times n_3(v)+m_3(v)<d(v)$.

By Lemma \ref{number}, we have $n_3(v)\leq\frac{d(v)}{2}$,  
$n_5(v)\leq d(v)-2n_3(v)-m_3(v)-1$, 
and $m_5(v)\leq d(v)-2\times n_3(v)-m_3(v).$

It follows from \textit{\textbf{CASE 7.2}} that
\begin{align*}
\mu^*(v)&\geq-n_3(v)\times\frac{2}{3}
+(\frac{5}{4}\times d(v)-\frac{11}{2}-n_3(v)\times\frac{1}{6})\\
&=\frac{5}{4}\times d(v)-\frac{11}{2}-n_3(v)\times\frac{5}{6}\\
&\geq \frac{5}{4}\times d(v)
-\frac{11}{2}-\frac{d(v)}{2}\times\frac{5}{6}\\
&=\frac{5}{6}\times d(v)-\frac{11}{2}\\
&>0 \indent\text{ for each }d(v)\geq 7.
\end{align*}

\indent Finally,  it  follows from all cases 
that $\sum_{x\in V(G)\cup F(G)}\mu^*(x)>0$, a contradiction. 
This completes the proof.



\end{document}